\documentclass[12pt]{amsart}

\makeatletter
\@namedef{subjclassname@2020}{\textup{2020} Mathematics Subject Classification}
\usepackage{youngtab}
\makeatother

\newcommand{\kp}{\kappa}

\newcommand{\ka}{\equiv_{A}}
\newcommand{\kc}{\equiv_{C}}

\newcommand{\ktkl}{\kt_{k}(\l)}

\newcommand{\sun}{\text{Sun}\,}

\newcommand{\ssyt}{\te{SSYT}}
\newcommand{\equ}{\equiv}

\newcommand{\oli}{\ol{1}}
\newcommand{\olii}{\ol{2}}
\newcommand{\oliii}{\ol{3}}

\newcommand{\row}{\text{row}\,}
\newcommand{\col}{\text{col}\,}
\newcommand{\colt}{\text{col}\,T}
\newcommand{\rowt}{\text{row}\,T}
\renewcommand{\lg}{*(lightgray) }

\newcommand{\mg}{\infty}

\newcommand{\qte}[1]{\q\te{#1}}

\renewcommand{\b}{\mathbf{b}}\renewcommand{\r}{\mathbf{r}}

\newcommand{\R}{\mathcal{R}}

\newcommand{\dsum}{\di\sum}

\newcommand{\down}{\downarrow}

\newcommand{\wh}{\ensuremath{\widehat}}
\newcommand{\el}{\ensuremath{\ell}}

\DeclareFontFamily{U}{mathx}{\hyphenchar\font45}
\DeclareFontShape{U}{mathx}{m}{n}{
      <5> <6> <7> <8> <9> <10>
      <10.95> <12> <14.4> <17.28> <20.74> <24.88>
      mathx10
      }{}
\DeclareSymbolFont{mathx}{U}{mathx}{m}{n}
\DeclareFontSubstitution{U}{mathx}{m}{n}
\DeclareMathAccent{\widecheck}{0}{mathx}{"71}

\newcommand{\tw}{\textwidth}

\newcommand{\lr}{\longleftrightarrow}

\renewcommand{\kill}[1]{}
\newcommand{\dummy}[1]{\mbox{}}

\makeatletter
\newcommand{\xequal}[2][]{\ext@arrow 0055{\equalfill@}{#1}{#2}}
\def\equalfill@{\arrowfill@\Relbar\Relbar\Relbar}
\makeatother

\newcommand{\mto}{\mapsto}

\newcommand{\ku}{\ensuremath{\emptyset}}

\renewcommand{\k}{\ensuremath{\ol{\mathrm{P}}}}
\newcommand{\n}{\ensuremath{\bm{n}}}

\newcommand{\hou}[3]{{#1}\equiv {#2}\pmod{#3}}

\newcommand{\vf}{\vfill}

\newcommand{\kyo}[1]{{\eh{#1}}}

\newcommand{\ok}{}

\newcommand{\m}{\ensuremath{\infty}}

\renewcommand{\k}[1]{\ensuremath{\left({#1}\right)}}

\newcommand{\ds}{\dots}

\newcommand{\bca}{\begin{cases}}
\newcommand{\eca}{\end{cases}}

\newcommand{\A}{\mathcal{A}}

\newcommand{\s}[1]{\ensuremath{\di\int{#1}\,dx}}

\renewcommand{\ss}[3]{\ensuremath{\di\int_{#1}^{#2}{#3}\,dx}}

\newcommand{\bpic}{\begin{picture}}\newcommand{\epic}{\end{picture}}

\newcommand{\beda}{\begin{edaenumerate}}
\newcommand{\eeda}{\end{edaenumerate}}

\newcommand{\g}{\ensuremath{\mathbf{g}}}

\newcommand{\cd}{\cdots}

\newcommand{\xto}{\xrightarrow}

\newcommand{\sh}[1]{\shadowbox{#1}}
\newcommand{\st}{\strut}

\newcommand{\shape}{\text{shape\,}}

\newcommand{\q}{\quad}

\newcommand{\up}{\uparrow}

\newcommand{\too}{\longrightarrow}

\newcommand{\from}{\longleftarrow}
\renewcommand{\from}{\,\reflectbox{$\to$}\,}

\newcommand{\bq}{\begin{quote}}\newcommand{\eq}{\end{quote}}

\newcommand{\ti}{\times}
\newcommand{\ra}{\rangle}\newcommand{\la}{\langle}

\newcommand{\be}{\begin{enumerate}}\newcommand{\ee}{\end{enumerate}}
\newcommand{\bce}{\begin{center}}\newcommand{\ece}{\end{center}}
\newcommand{\bde}{\begin{description}}\newcommand{\ede}{\end{description}}
\newcommand{\bri}{\begin{flushright}}\newcommand{\eri}{\end{flushright}}
\newcommand{\bb}{\begin{block}}\newcommand{\eb}{\end{block}}
\newcommand{\bt}{\begin{thm}}\newcommand{\et}{\end{thm}}
\newcommand{\bpf}{\begin{proof}}\newcommand{\epf}{\end{proof}}
\newcommand{\bex}{\begin{ex}}\newcommand{\eex}{\end{ex}}
\newcommand{\bexr}{\begin{exr}}\newcommand{\eexr}{\end{exr}}
\newcommand{\bft}{\begin{fact}}\newcommand{\eft}{\end{fact}}
\newcommand{\brk}{\begin{rmk}}\newcommand{\erk}{\end{rmk}}
\newcommand{\ba}{\begin{align*}}\newcommand{\ea}{\end{align*}}
\newcommand{\bexe}{\begin{exe}}\newcommand{\eexe}{\end{exe}}
\newcommand{\tn}{\textnormal}

\newcommand{\bit}{\begin{itemize}}\newcommand{\eit}{\end{itemize}}

\newcommand{\bcm}{}
\newcommand{\ol}{\overline}\newcommand{\ul}{\underline}
\newcommand{\hf}{\hfill}

\newcommand{\cc}{\ensuremath{\mathbf{C}}}
\newcommand{\nn}{\ensuremath{\mathbf{N}}}

\newcommand{\rr}{\ensuremath{\mathbf{R}}}
\newcommand{\zz}{\ensuremath{\mathbf{Z}}}

\newcommand{\bd}{\begin{defn}}\newcommand{\ed}{\end{defn}}
\newcommand{\bp}{\begin{prop}}\newcommand{\ep}{\end{prop}}

\newcommand{\eh}{\emph}\newcommand{\al}{\alpha}
\newcommand{\sub}{\subseteq}
\newcommand{\lam}{\lambda}
\renewcommand{\i}{\item}\newcommand{\fb}{\fbox}
\newcommand{\mb}{\mbox}
\newcommand{\te}{\text}\newcommand{\ph}{\phantom}
\newcommand{\wt}{\widetilde}\newcommand{\sm}{\setminus}

\renewcommand{\l}{\left}\renewcommand{\r}{\right}

\newcommand{\then}{\Longrightarrow}

\newcommand{\di}{\displaystyle}\renewcommand{\a}{\ensuremath{\bm{a}}}
\renewcommand{\b}{\ensuremath{\bm{b}}}

\newcommand{\e}{\ensuremath{\bm{e}}}

\newcommand{\x}{\ensuremath{\bm{x}}}
\newcommand{\y}{\ensuremath{\bm{y}}}

\newcommand{\np}{\newpage}

\renewcommand{\b}{\beta}

\renewcommand{\a}{\alpha}

\renewcommand{\up}{\uparrow}

\renewcommand{\x}{\mathbf{x}}
\renewcommand{\y}{\mathbf{y}}

\renewcommand{\s}{\sigma}

\usepackage[dvipdfmx]{graphicx}
\usepackage[dvipsnames]{xcolor}
\usepackage{float,afterpage}
\usepackage{asymptote,layout}
\usepackage{wrapfig,epic}
\usepackage{amsmath,amssymb,cases,pict2e}

\usepackage{multicol}
\usepackage{tikz}
\tikzset{
    cell/.style={
        anchor=south west,
        draw,
        minimum size=1cm,
    },
}

\usepackage{geometry}
\usepackage{exscale,latexsym,bm}
\usepackage{amssymb,enumerate,amsmath,amsthm,amsfonts}
\usepackage{verbatim,fancybox,wasysym,fancyhdr,type1cm}
\usepackage[frame,all,poly,curve,knot,arrow]{xy}
\usepackage{ytableau}
\usepackage{colortbl}
\usepackage{boxedminipage}
\usepackage{multirow}

\graphicspath{{./figs/}}
\theoremstyle{definition}
\newtheorem{thm}{Theorem}[section]

\newtheorem{lem}[thm]{Lemma}
\newtheorem{prop}[thm]{Proposition}\newtheorem{cor}[thm]{Corollary}

\newtheorem{exr}[thm]{Exercise}
\newtheorem{ob}[thm]{Observation}

\newtheorem{ex}[thm]{Example}

\newtheorem{defn}[thm]{Definition}\newtheorem{rmk}[thm]{Remark}
\newtheorem{fact}[thm]{Fact}
\newtheorem{block}[thm]{}
\newtheorem*{exe}{Exercise}

\renewcommand{\a}{\alpha}

\renewcommand{\l}{\lam}

\renewcommand{\g}{\gamma}

\newcommand{\vp}{\zeta_{1}}

\newcommand{\innn}{\in\nn}
\renewcommand{\g}{\Gamma}

\newcommand{\rect}{\text{Rect}\,}
\newcommand{\rects}{\text{Rect}\,S}

\renewcommand{\R}{\mathbf{R}}
\usepackage{youngtab}

\newcommand{\ot}{\otimes}

\newcommand{\size}{\text{size\,}}
\newcommand{\inside}{\text{inside\,}}
\newcommand{\out}{\text{outside\,}}

\renewcommand{\l}{\lambda}
\renewcommand{\wt}{\text{wt}}
\renewcommand{\ep}{\varepsilon}
\newcommand{\yd}{\ydiagram}

\renewcommand{\r}{{\bf r}}
\renewcommand{\vp}{\varphi}
\newcommand{\ttt}[3]{\ensuremath{
\begin{ytableau}{#1}&{#2}\\{#3}\end{ytableau}}
}

\renewcommand{\g}{\gamma}

\renewcommand{\m}{\mu}
\newcommand{\kt}{\te{KT}}
\newcommand{\olk}{\ol{k}}
\renewcommand{\A}{\mathbf{A}}
\newcommand{\B}{\mathbf{B}}

\newcommand{\ssot}{\te{SSOT}}

\renewcommand{\st}{\text{st}}
\renewcommand{\ot}{\text{OT}}
\newcommand{\OT}{\text{OT}}
\renewcommand{\rr}{\equ_{C}}
\renewcommand{\a}{\bm{a}}

\newcommand{\pmk}{[\olk]}
\newcommand{\C}{\mathbf{C}}
\renewcommand{\sh}{\te{sh\,}}
\newcommand{\sha}{\ensuremath{\te{sh\,}_{A}}}
\newcommand{\shc}{\ensuremath{\te{sh\,}_{C}}}
\renewcommand{\ss}{\text{ss}}

\renewcommand{\a}{\al}
\renewcommand{\B}{\mathbf{B}}

\renewcommand{\a}{\alpha}

\newcommand{\wts}{\ensuremath{\text{wt}^{*}\,}}
\renewcommand{\b}{\beta}
\newcommand{\li}{\la i\ra}
\newcommand{\kot}{\te{KOT}}
\renewcommand{\sh}{\te{sh\,}}
\newcommand{\Sh}{\te{Sh\,}}

\newcommand{\olai}{\ol{i}}
\newcommand{\olaii}{\ol{i+1}}
\newcommand{\fbi}{\fb{$i$}}
\newcommand{\fbib}{\fb{$\ol{i}$}}
\newcommand{\fbibar}{\fb{$\ol{i}$}}
\newcommand{\fbip}{\fb{${i+1}$}}

\newcommand{\fba}{\fb{$A$}}
\newcommand{\fbb}{\fb{$B$}}

\renewcommand{\L}{\mathbf{L}}
\renewcommand{\sh}{\te{sh}\,}
\newcommand{\oob}{1\ol{1}}
\newcommand{\ttb}{2\ol{2}}
\newcommand{\hako}{\te{Box\,}}
\newcommand{\fbx}{\fb{$x$}}
\newcommand{\fby}{\fb{$y$}}
\newcommand{\fbz}{\fb{$z$}}
\renewcommand{\m}{\mu}

\newcommand{\tsp}{\te{sp}}
\renewcommand{\l}{\lambda}
\newcommand{\red}{\text{red}}
\newcommand{\clmn}{c_{\l\m}^{\n}}
\newcommand{\rh}{\rho}

\begin{document}

\title{RSK correspondence for King tableaux with Berele insertion}


\author{Masato Kobayashi}
\thanks{Corresponding author: 
Masato Kobayashi}

\author{Tomoo Matsumura}
\thanks{Tomoo Matsumura is supported partially by JSPS Grant-in-Aid for
Scientific Research (C) 20K03571 and (B) 23K25772.
}

\date{\today}                                       



\address{Masato Kobayashi\\
Department of Engineering\\
Kanagawa University, Rokkaku-bashi, Yokohama, Japan.}
\email{masato210@gmail.com}

\address{Tomoo Matsumura\\
Department of Natural Sciences\\
International Christian University, 
Tokyo, Japan.
}
\email{matsumura.tomoo@icu.ac.jp}

\maketitle
\begin{abstract}
We establish a bijective RSK correspondence of type C for King tableaux with Berele insertion as a reformulation of Sundaram's correspondence (1986). 
For its $Q$-symbol, we make use of semistandard oscillating tableaux (SSOT), a new object which 
S.J.Lee (2025) introduced. Further, we show hidden duality of Cauchy identity through RSK correspondences of type A and C. Finally, we prove that the generating function of SSOT is symmetric by constructing a new sort of Bender-Knuth involution.
\end{abstract}
\tableofcontents


\ytableausetup{centertableaux}
\section{Introduction}

\subsection{Combinatorics of Young tableaux}

\eh{Young tableaux} play a significant role in algebraic combinatorics. 
Here are several classic topics.
\begin{itemize}
	\item row/column-insertion
	\item Knuth equivalence
	\item skew tableaux, jeu de taquin 
	\item Robinson-Schensted (RS) correspondence
	\item Robinson-Schensted-Knuth (RSK) correspondence
	\item Schur function
	\item Bender-Knuth involution
	\item Cauchy identity
	\item Littlewood-Richardson Rule 
	\item Kashiwara's crystal
\end{itemize}
Fulton discussed most of these topics in his book \cite{fu}. Indeed, there are many applications of these ideas 
to other areas, such as 
algebra, representation theory and geometry.

\subsection{King (1976), Berele (1986), Sundaram (1986), S.J.Lee (2025)}
There are many other kinds of tableaux.
One is \eh{King tableaux} ({symplectic tableaux}) as King introduced \cite{ki}. Afterward Berele \cite{be} introduced a row-insertion algorithm (now known as \kyo{Berele insertion}) and found a bijective RS  correspondence for such tableaux. This is a natural extension of the 
formalism in type A. 
Further, Sundaram \cite{sun1, sun2} discussed 
certain correspondence among King tableaux, Burge arrays and Knuth equivalence. 
More recently, S.J.Lee \cite{lee} introduced 
\eh{semistandard oscillating tableaux} (SSOT), a newcomer to this area, as 
a sequence of partitions with several boxes added or deleted  at each step (and certain numbering to those). 
He further proved that 
there is a crystal-weight-preserving bijection 
between certain SSOTs and King tableaux. 
Motivated by these authors, we aim in this article 
to develop combinatorics of King tableaux.

\subsection{Main results}
Here are our main results.
\begin{itemize}
	\item 
Characterizations of type C Knuth equivalence 
(Theorem \ref{t4}, Corollary \ref{ct1}) 
	\item Type C RSK correspondence 
	(Theorem \ref{t5})
	\item Duality of Cauchy identity (Corollary \ref{c1})
	\item Bender-Knuth involution for SSOTs (Lemma \ref{lab})
	\item Symmetry of SSOT function (Theorem \ref{t6})
\end{itemize}
They not only unfold combinatorics in type C theory individually but reveal 
its rich interaction with type A theory.
\begin{rmk}
After writing this manuscript, we found that 
Roby-Terada (2005) \cite{rt} discussed standardization of a word, a King tableau and Berele insertion, objects on $P$-side. Our key idea in this article is 
its dual, semistandardization of objects on $Q$-side. 
\end{rmk}

\begin{rmk}
In the type C tableaux theory, we often use letters 
$1, \ol{1}, \ds, k, \ol{k}$ for entries of tableaux with certain total order. Other than King tableaux, there is actually another style, \kyo{Kashiwara-Nakashima} (KN) tableaux as Kashiwara-Nakashima (1994) \cite{kn}, 
Lecouvey (2002) \cite{le}, Sheats (1999) \cite{sh} discussed, for example.
The total order for King tableaux is 
\[
1<\oli<2<\olii<\cd <k<\ol{k}
\]
while 
the one for KN is 
\[
1<2<\cd <k<\ol{k}<\cd <\olii<\oli.
\]
We also remark that 
there are many ways to read 
row and column words from a tableau.
In this article, we follow Fulton's style \cite{fu} 
as we make sure later on.
\end{rmk}


\subsection{Organization of the article}

In Section 2, we briefly review 
definitions and results in type A theory on SSYTs.
In Section 3, we confirm key ideas and facts on 
King tableaux and Berele insertion. 
In Section 4, we give a proof of 
most of our main results.

\subsection*{Acknowledgment}
The authors would like to 
thank the anonymous referee for helpful comments to improve the manuscript.


\section{type A theory}
Throughout $k\innn$ and let 
$[k]=\{1, 2, \ds, k\}$. 
A \emph{partition} (Young diagram) $\l$ is a finite weakly increasing  sequence of positive integers. 
The symbol $\el(\l)$ means the number of its rows. 
A \kyo{Young tableau} $T$ is a filling of each box on a Young diagram $\l$ (English style) 
by an element of $[k]$. 
We write $T(m, n)$ for the entry of $T$ at row $m$ and column $n$. To indicate its shape, we write as $\sh T=\l$. 

\subsection{Young tableaux}


\begin{defn}
A tableau $T$ is \emph{semistandard} if it satisfies the following.
\begin{enumerate}
\item Entries are weakly increasing along the rows;
\item Entries are strictly increasing down the columns. (column-strict condition)
\end{enumerate}
Denote by $\ssyt_{k}(\l)$ the set of 
semistandard Young tableaux 
of entries in $[k]$ of shape $\l$. 
For convenience, set 
\[
\ssyt_{k}=\bigcup_{\l}\ssyt_{k}(\l),
\ssyt(\l)=\bigcup_{k}\ssyt_{k}(\l),
\ssyt =\bigcup_{k}\ssyt_{k}.
\]
A semistandard tableau $T$ of shape $\l$ is \kyo{standard} if 
entries are strictly increasing along the rows 
and each of $1, \ds, |\l|$ appears exactly once.
For the set of such tableaux, use symbol $\te{ST}(\l)$.
\end{defn}

\subsection{row insertion}

We review type A \emph{Schensted row-insertion}.
Let $T$ be a tableau and $x$ a letter.

\begin{enumerate}
\item Insert $x=x_{1}$ to the first row 
of $T$ either displacing (bump) the smallest number which is larger than $x_{1}$ (say this number is $x_{2}$). 
If no number is larger than $x$, just add it to the end of the first row (and then stop).
\item Insert $x_{2}$ to the second row with the same procedure.
\item Continue this until we stop.
\end{enumerate}
Denote the resulting tableau by $T\from x$.

\begin{quote}
For example, $T=
\begin{ytableau}
 1 &1 &3&4 \\
 2&5\\
 4\\
\end{ytableau}$,
$T\from 2=
\begin{ytableau}
 1 &1 &2&4 \\
 2&3\\
 4&5\\
\end{ytableau}.$
\end{quote}

Now let $[k]^{*}$ denote the set of 
all finite words of $[k]$. 
Its subset $[k]^{n}$ consists of 
all such words of length $n$.
\begin{defn}
For an SSYT $T$, define $\row T$ by reading \emph{row  word} from bottom to top, left to right.
\end{defn}

\begin{fact}[{\cite[p.17]{fu}}]
$\row:\ssyt_{k}\to [k]^{*}, T\mto \row T$ is injective.
\label{rin}
\end{fact}

For partitions $\l, \mu$ such that $\l\sub \m$ as sets of boxes, 
define a \kyo{skew shape} $\m\sm\l$ to be the set of boxes in $\m$ but not in $\l$.
A \kyo{skew tableau} is a filling of a skew shape. 
It makes sense to talk about (semi)standardness and its row word for such tableaux.
A \emph{punctured tableau} is 
a filling of a skew shape with the exactly one box empty. 

A \emph{slide} on a punctured tableau is 
a local operation to move the empty box to its below or right shown under the following rule.
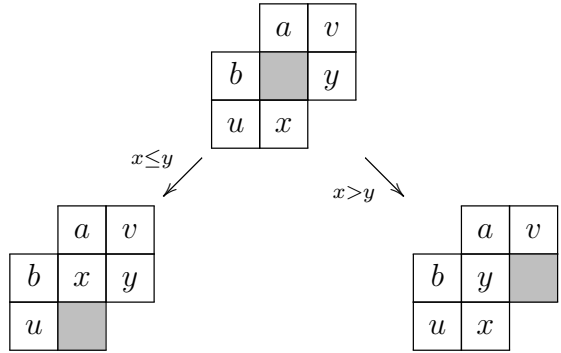
\begin{center}
\begin{minipage}[c]{.8\tw}
\xymatrix@=5mm{
*{}&
*+{\begin{ytableau}
	\none&a&v\\
	b&\lg{}$\mb{\ph{a}}$&y\\
	u&x\\
\end{ytableau}}
\ar@{->}_{x\le y}[ld]
\ar@{->}_{x> y}[rd]
&*{}\\
*+{\begin{ytableau}
	\none&a&v\\
	b&x&y\\
	u&\lg{}\mb{\ph{a}}\\
\end{ytableau}}
&*{}&
*+{\begin{ytableau}
	\none&a&v\\
	b&y&\lg{}\mb{\ph{a}}\\
	u&x\\
\end{ytableau}}
}
\end{minipage}
\end{center}
We can repeat such operations 
as follows.
\begin{itemize}
\item If the empty box is at a corner, we just stop.
\item Otherwise, 
there is a box below or right to the empty box.	
If only one of them exists, then slide it in this direction.
If both exist, move the empty box to the one including a smaller letter (If those letters are equal, \kyo{below} is chosen).
\end{itemize}

Continuing this, the empty box eventually arrives at a corner.
The whole process of a series of slidings on a skew tableau is \emph{jeu de taquin}. 


\subsection{RS correspondence}

Given $w=w_{1}\cd w_{n}\in [k]^{*}$, 
we construct $(P_{j}(w), Q_{j}(w))_{j=0}^{n}$ 
inductively as follows.
Set $P_{0}(w)=\ku, Q_{0}(w)=\ku$. 
For $j\ge 1$, 
let $P_{j}(w)=P_{j-1}(w)\from w_{j}$. 
To construct $Q_{j}(w)$, 
add one box $\fbox{$j$}$ to $Q_{j-1}(w)$ at the  position which the new box of $P_{j}(w)$ appeared. 
Then set $P(w)=P_{n}(w), Q(w)=Q_{n}(w)$.
Call $P(w)$ its \emph{inserting tableau} and $Q(w)$ \emph{recording tableau}.
\begin{thm}[RS correspondence of type A]\label{t1}
\[
\Phi:
[k]^{*}_{n}\to 
\bigcup_{|\l|=n}
\ssyt_{k}(\l)\ti \text{ST}(\l),
\q 
w\mto (P(w), Q(w))
\]
is a bijection.
\end{thm}


\subsection{Knuth equivalence}
Knuth \cite{knu} introduced certain transformations for  a word. 
Let $x, y, z$ be letters and $w, w'$ words. 
Consider the transformations each of which moves only three adjacent letters $x, y, z$ in a word as follows.
For $x<y\le z$:
\begin{align*}
	(K_{1})\ph{a}& yxz\mto yzx&& (K_{2})\ph{a} yzx\mto yxz
\end{align*}
For $x<y\le z$: 
\begin{align*}
	(K_{3})\ph{a}& xzy\mto zxy&& 
	(K_{4})\ph{a} zxy\mto xzy 
\end{align*}
Say $w, w'$ are \emph{elementary Knuth equivalent} if 
we obtain $w'$ from $w$ by one of $K_{1}, K_{2}, K_{3}, K_{4}$.
Say $w, w'$ are \emph{Knuth equivalent} ($w\equiv w'$ or 
$w\equiv_{A} w'$ for emphasis) 
if there is a sequence of words 
$w_{0}=w, w_{1}, \ds, w_{n}=w'$ such that 
each $w_{j}$ is elementary Knuth equivalent 
to $w_{j+1}$.




\begin{defn}
Define $\colt$ by reading \emph{column word} of $T$, 
from bottom to top, left to right.
\end{defn}
Note that 
$\col:\ssyt_{k}\to [k]^{*}$ is also injective. 
Below, we collect key results for reader's convenience.
For details, see \cite[Chapters 1, 2]{fu}.

\begin{fact}
For any SSYT $T$, $\rowt\equiv_{A}\col T$.
\end{fact}


\begin{thm}\label{th}
For an SSYT $T$ and $x$,
\[
\row(T\from_{A} x)\ka(\row T)x.
\]\end{thm}

\begin{thm}
For a word $w$, there exists a unique SSYT $T$ such that 
\[
w\equiv_{A} \rowt.
\qte{(Indeed, $T=P(w)$)}
\]
\end{thm}

\begin{thm}\label{t2}
$w{\equiv}_{A}w'\iff P(w)=P(w').$
\end{thm}

\begin{fact}
Playing jeu de taquin from a skew tableau $S$ with all choices of inner corners leads to the same SSYT.
\end{fact}
We write this SSYT, the \kyo{rectification} of $S$, as $\rect{S}$.

\begin{fact}
Let $S_{1}, S_{2}$ be skew tableaux. Then 
$\row S_{1}\ka\row S_{2}$
$\iff \rect S_{1}=\rect S_{2}.$
\end{fact}

%
%

%

%
%

\section{type C theory}

\subsection{King tableaux\ok}

Consider barred numbers $\oli, \ds, \olk$ 
and let $[\ol{k}]=\{1, \ol{1}, \ds, k, \ol{k}\}$.
Define the \emph{symplectic order} 
\[
1<\ol{1}<2<\ol{2}<\cd<k<\ol{k}.
\]


\begin{defn}\ok
A \emph{King} \emph{tableau} (KT) of shape $\l$ is a filling of the Young diagram of $\l$ on the alphabet $[\olk]$ 
with all of the following.
\begin{enumerate}
\item Entries are weakly increasing along rows.
\item Entries are strictly increasing down columns.
\item All entries in row $i$ are $\ge i$ (symplectic condition).
\end{enumerate}
Denote by $\kt_{k}(\l)$ the set of King tableaux 
of entries $[\olk]$ of shape $\l$. 
Also let $\kt_{k}=\cup_{\l}\kt_{k}(\l)$
and $\kt=\cup_{k}\kt_{k}$. 
Besides, 
$\ssyt_{\olk}$ means the set of 
all SSYTs of entries $[\olk]$.
\end{defn}
\begin{rmk}
Some authors call such 
\emph{symplectic tableaux}. 
\end{rmk}
\begin{ex}
We see 
$T=
\ytableausetup{centertableaux}
\begin{ytableau}
 1 & 1 & 2 &{\ol{3}}\\
 2& \ol{2} &\ol{2}   \\
\ol{3}  
\end{ytableau}
\in \kt_{3}$
while 
$U=
\ytableausetup{centertableaux}
\begin{ytableau}
 1 & 1 & 2 &{\ol{3}}\\
 2& \ol{2} &\ol{2}   \\
\lg\ol{2}  
\end{ytableau}$
is merely in $\ssyt_{\ol{3}}$ with violation of symplectic condition at the first column. 
\end{ex}

\subsection{Berele insertion\ok}

We have seen (row-)insertion 
 $T\from x$ for $T\in \ssyt$  and $x\in [k]$. 
Let us write this also as $T\from_{A} x$ whenever we wish to emphasize it. As analogy of this, 
we are going to define \emph{Berele insertion} 
$T\from x$ (still the same notation, or $T\from_{C} x$ if necessary) for a KT $T$ and $x\in [\olk]$. 

Define $T\from x$ with a series of row-insertions of type A except the following case:

If row $i$ contains at least one $\ol{i}$, 
and we row-insert $i$ to this row, then, instead of adding $i$ and bumping $\ol{i}$ to the $(i+1)$-st row, 
\begin{enumerate}
\item replace the first $\ol{i}$ to $i$,
\item replace the first $i$ to empty box,
\end{enumerate}
\[
\te{row $i$:\,}
\begin{ytableau}
 i &i &\cd&i&\ol{i}&\ol{i}&\cd
\end{ytableau}
\to 
\begin{ytableau}
\lg{}\mb{}&i &\cd&i&{i}&\ol{i}&\cd
\end{ytableau}
\]
and play jeu de taquin (as an $\ssyt$ of entries $[\olk]$) with this empty box. 
As Berele showed \cite{be}, the resulting tableau is always King.
\begin{ex}\label{exb}
Let 
$
T=
\begin{ytableau}
	1&2&4\\
	\ol{2}&\ol{2}&6\\
	3&\ol{4}\\
	5\\
\end{ytableau}$, 
$x=\oli$. 
Inserting $x$ to the first row bumps 2 to the second. 
By definition of Berele insertion, 
this 2 kills the leftmost $\ol{2}$ in the first column to  produce an empty box; 
say this punctured tableau is $T_{1}$.
Sliding it twice to a corner and 
forget it after all, we have $T\from_{C}x$.
\[
T_{1}=
\begin{ytableau}
	1&\ol{1}&4\\
	\lg&\ol{2}&6\\
	3&\ol{4}\\
	5\\
\end{ytableau}
\xto[\te{slide}]{}
\begin{ytableau}
	1&\ol{1}&4\\
	\olii&\ol{4}&6\\
	3&\lg{}\\
	5\\
\end{ytableau}
\xto[]{}
\begin{ytableau}
	1&\ol{1}&4\\
	\olii&\ol{4}&6\\
	3\\
	5\\
\end{ytableau}=T\from_{C}x.
\]
Alternatively, we can regard this process as follows: 
First, consider $T\from_{A}x$ instead 
and replace the domino violating symplectic condition by two empty boxes; say we got $T_{2}$. 
\[
\begin{ytableau}
	1&\oli&4\\
	2&\ol{2}&6\\
	\ol{2}&\ol{4}\\
	3\\
	5\\
\end{ytableau}
\too
\begin{ytableau}
	1&\oli&4\\
	\lg{}&\ol{2}&6\\
	\lg{}&\ol{4}\\
	3\\
	5\\
\end{ytableau}=T_{2}.
\]
Now, playing jeu de taquin from $T_{2}$ also results in 
$T\from_{C} x$. 
This is because two skew tableaux below the first row of $T_{1}, T_{2}$
\[
\begin{ytableau}
	\lg{}&\ol{2}&6\\
	3&\ol{4}\\
	5\\
\end{ytableau}
\q\text{ and }
\q\begin{ytableau}
	\lg{}&\ol{2}&6\\
	\lg{}&\ol{4}\\
	3\\
	5\\
\end{ytableau}
\]
with identical row word $53\ol{4}\ol{2}6$
end up with the same SSYT after jeu de taquin.
\end{ex}

\subsection{type C RS correspondence\ok}

Here, we review type C RS (or Berele-Robinson-Schesnted) correspondence. To describe an appropriate $Q$-symbol for this case, we need several terms and definitions. For partitions $\l$ and $\mu$, 
write $\l\lhd \mu$ if $\l\subset \mu$ and 
$|\mu|=|\l|+1$.

\begin{defn}\ok 
An \emph{oscillating tableau} (OT) $S$ of shape $\l$ of length $n$ is a sequence of partitions 
\[
S=(\l_{0}, \l_{1}, \ds, \l_{n})
\]
such that $\l_{0}=\ku$, $\l_{n}=\l$ and for each $j$, 
$\l_{j}\lhd \l_{j+1}$ or $\l_{j+1}\lhd \l_{j}$. 
It is a $k$-OT 
if each $\l_{j}$ has at most $k$ rows. 
Denote by $\OT_{k,n}(\l)$ the set of such OTs.
\end{defn}
\begin{rmk}
Some authors call such \eh{up-down tableaux}.
\end{rmk}

In the sequel, a word $w$ means an element of $ [\olk]^{*}$, the set of all words in $[\olk]$ unless otherwise mentioned. 
For a word $w=w_{1}\cd w_{n}$, we construct a pair $(P(w), Q(w))\in \kt_{k}(\l)\ti \OT_{k,n}(\l)$ as follows. 
Start with 
$P_{0}(w)=\ku, Q_{0}(w)=\ku$. For $j\ge 1$, define 
\[
P_{j}(w)=P_{j-1}(w)\from_{C} w_{j}, \q Q_{j}(w)=\text{sh\,} P_{j}(w). 
\]
Finally, let 
$P(w)=P_{n}(w), Q(w)=(Q_{j}(w))_{j=0}^{n}$. 
\begin{thm}[RS correspondence of type C \cite{be}]\ok
\label{t3}
\[
\Phi_{\te{Ber}}:
[\olk]_{n}^{*} \to 
\bigcup_{\el(\l)\le k}
\kt_{k}(\l)\ti \OT_{k,n}(\l), 
\q 
w\mto (P(w), Q(w))
\]
is a bijection.
\end{thm}


\begin{ex}\ok
For $w=\olii 2 \olii 2 1 \oli $, 
Berele insertions go as follows.
\[
\begin{ytableau}
\olii
\end{ytableau}
\to 
\begin{ytableau}
2\\
\olii
\end{ytableau}
\to 
\begin{ytableau}
2&\olii\\
\olii
\end{ytableau}
\to  
\begin{ytableau}
2&2\\
\olii&\olii
\end{ytableau}
\to
\begin{ytableau}
1&2\\
\olii
\end{ytableau} 
\to 
\begin{ytableau}
1&\ol{1}
\end{ytableau}
\]
\[
(P(w), Q(w))=
\k
{\,\begin{ytableau}
1&\ol{1}
\end{ytableau}\,, 
\k{
\yd{1}\,,\yd{1,1}\,,\yd{2, 1}\,, \yd{2, 2}\,, 
\yd{2, 1}\,, \yd{2} 
\,}
}
\]
\end{ex}

\begin{rmk}
For convenience, we 
write $P_{A}(w)$ and $P_{C}(w)$ 
for the SSYT and the KT constructed by type A and type C insertions from $w$, respectively. 
In Sundaram's notation \cite{sun1}, 
\[
(\ku\from w)_{1}=P_{A}(w), \q
(\ku\from w)_{2}=Q_{A}(w), \q
(\ku
\overset{\mathcal{B}}{\longleftarrow}w)_{1}=P_{C}(w).
\]
\end{rmk}

\subsection{type C Knuth equivalence\ok}

How can we introduce type C Knuth equivalence for words? Obviously, we need to understand what happens  to a word when a Belere insertion deletes a box. 
As seen in Example \ref{exb}, in such a case, two 
empty boxes arise in the first column. 
Hence we should discuss words coming from a single column.

\begin{defn}\ok
A word 
$c=c_{n}\cd c_{1}$ is a \emph{column word} if 
$c_{i+1}>c_{i}$ for all $i$. 

For column words $c, c'$, 
define $c\sim c'$ if 
$P_{C}(c)=P_{C}(c')$.
\end{defn}

\begin{ex}
Let 
\[
c=\ol{7}7\ol{6}{6}\ol{5}\ol{4}43\ol{2}2,
\q 
c'=\ol{5}3\ol{2}2.
\]
Then $P_{C}(c)=P_{C}(c')$ and thus $c\sim c'$.
\end{ex}

\begin{defn}
Define $w\equiv_{5}w'$
if there exist some word $v$ and 
column words $c, c'$ 
such that 
$w\ka cv, w'\ka c'v$ 
and $c\sim c'$. 
In such a case, say that we obtained $w$ from $w'$ (or the other way around) by $K_{5}$ transformation. (Observe that $P_{C}(w)=P_{C}(w')$)
\end{defn}

\begin{defn}
Recall that we introduced 
$K_{j}$ transformation $(j=1, \ds, 5)$. $w$ and $w'$ are \kyo{type C Knuth equivalent} 
($w\kc w'$) 
if there exists a sequence of words 
$w_{0}, w_{1}, \ds, w_{n}$ such that 
$w_{0}=w, w_{n}=w'$ and 
$w_{i+1}$ is obtained from $w_{i}$ by some $K_{j}$ transformation.
\end{defn}
\begin{rmk}\ok
Clearly, $w\equiv_{A} w'\then w\rr w'$ 
since we can move $w$ to $w'$ 
with only $K_{j}$ transformations, $j=1, 2, 3, 4$.
\end{rmk}

\begin{lem}\label{l1}\ok
For any KT $T$ and a letter $x$, 
\[
\row(T\from_{C} x)\rr (\rowt)x.
\]
\end{lem}
\begin{proof}
If $T\from_{C} x$ adds a box, 
then 
\[
\row(T_{C}\from x)= \row(T_{A}\from x)
\equ_{A}(\rowt)x
\]
and thus 
$\row(T\from_{C} x)\equ_{C}(\rowt)x$.
Suppose next that $T\from_{C} x$ deletes a box.
Then $T\from_{A} x$ must look like
\[
\begin{ytableau}
	\mb{}&\mb{}&\mb{}&\mb{}&\mb{}&\\
	\mb{}&\mb{}&\mb{}&\mb{}&\mb{}&\\
	\lg{}i&\mb{}&\mb{}&\\
	\lg{}\ol{i}&\mb{}&\\
	\mb{}&\\
	\mb{}&\\
\end{ytableau}
\]
for some (unique) $i$ with $i$ in row $i$ 
and $\ol{i}$ in row $i+1$. 
In other words, 
\[
\row(T\from_{A}x)\ka \col (T\from_{A}x)
=cv
\]
for some column word 
$c=c_{n}\cd c_{1}$, 
$(c_{i+1}, c_{i})=(\ol{i}, i)$ and a word $v$ 
for other columns $j\ge2$.
Now, by definition of Berele algorithm, 
$T\from_{C}x$ is exactly 
the unique SSYT whose column (row) word is 
$\ka c'v$ where $c'=c_{n}\cd c_{i+2}c_{i-1}\cd c_{1}$ (that is, $c\sim c'$). Thus, 
\[
\row (T\from_{C}x)\ka c'v,
(\row T)x\ka \row (T\from_{A}x)=cv,
\]
and $c\sim c'$ imply 
$(\row T)x\kc \row (T\from_{C} x)$.
\end{proof}

\begin{lem}\label{l2}
For any KT $T$, we have
$P_{C}(\row T)=T$.
\end{lem}
\begin{proof}
We prove 
$P_{C}(\row T)=T$ by induction on the number $r$ 
of rows of $T$. If $r=1$, 
say $T=
\begin{ytableau}
	x_{1}&\cd&x_{n}\\
\end{ytableau}$, $x_{i}\le x_{i+1}$. Clearly, 
\[
P_{C}(\row T)
=\begin{ytableau}
	x_{1}&\cd&x_{n}\\
\end{ytableau}=T.
\]
Suppose $r\ge 2$. Let $T'$ be the KT with 
rows $2, \ds, r$ of $T$ (satisfying symplectic condition automatically). Again, say the first row of $T$ is 
\[
\begin{ytableau}
	x_{1}&\cd &x_{n}\\
\end{ytableau},
\]
$x_{i}\le x_{i+1}$. By induction, we have 
\[
P_{C}(\row T)=P_{C}(\row T' x_{1}\cd  x_{n})=
((P_{C}(\row T') \from x_{1})\from \cd )\from x_{n}
\]
\[=((T' \from x_{1})\from \cd )\from x_{n}
=T.
\]
This proves the claim.
\end{proof}

\begin{thm}\label{t4}
For each $w\in [\olk]^{*}$, there exists a unique 
$T\in \kt_{k}$ such that 
\[
w\equ_{C}\row T.
\]
In particular, we have $T=P_{C}(w)$.
\end{thm}
\begin{proof}
Say $w=w_{1}\cd w_{n}$ and $P_{m}=P_{C}(w_{1}\cd w_{m})$. 
It follows from Lemma \ref{l1} that 
\begin{align*}
	\row P_{C}(w)&=
	\row(P_{n-1}\from w_{n})
	\\&\rr (\row P_{n-1})w_{n}
	\\&=(\row P_{n-2}\from w_{n-1})w_{n}
	\\&\rr (\row P_{n-2})w_{n-1}w_{n}
	\\&=\cd \rr w_{1}\cd w_{n-1}w_{n}=w.
\end{align*}
To show uniqueness of such a KT $T$, 
suppose $\row T_{1}\equ_{C}\row T_{2}$ for some  $T_{1}, T_{2}\in \kt_{k}$. 
It follows from Lemma \ref{l2} that length of $\row T_{1}$ and $\row T_{2}$ are equal.
Now regarding $T_{1}, T_{2}$ as SSYTs on the alphabet 
$[\olk]$, we have $\row T_{1}\equ_{A}\row T_{2}$. 
It follows from Fact \ref{rin} that 
$T_{1}=T_{2}$. 
\end{proof}

We have seen 
$w\equiv_{A} w'\iff P_{A}(w)=P_{A}(w')$ for type A words. Now, we can state an analogous result. 

\begin{cor}\label{ct1}
\ok
$w\rr w'\iff P_{C}(w)=P_{C}(w')$.
\end{cor}
\begin{proof}
$(\then:)$ 
It is enough to show $P_{C}(w)=P_{C}(w') $ if 
$w'$ is obtained from $w$ by $K_{j}$ transformation $(j=1, 2, 3, 4, 5)$. 
For $j=1, 2, 3, 4$, this is trivial since 
$w\equ_{A} w'$ implies 
$w\rr w'$.
Now suppose 
$w'$ is obtained from $w$ by $K_{5}$.
As seen above, we must have 
$P_{C}(w)=P_{C}(w')$ in such a case.
$(\reflectbox{$\then$})$:
Thanks to Theorem \ref{t4}, we have 
\[
w\rr \row P(w)=\row P(w')\rr w'.
\]
\end{proof}
\begin{ex}
Example \ref{exb} shows nothing but 
$53\ol{4}{\ol{2}}\ol{2}6124\ol{1}\kc
53\ol{2}\ol{4}61\ol{1}4$.
\end{ex}
\begin{rmk}
Sundaram \cite[Theorem 11.2]{sun1} has implicitly  described a relation between type A and type C Knuth equivalences: let $w, w'$ be words of the same length. 
She showed that $w\ka w'$ if and only if 
$w\kc w'$ and 
$(\ku\stackrel{\mathcal{B}}{\longleftarrow}w)_{3}=
(\ku\stackrel{\mathcal{B}}{\longleftarrow}w')_{3}$ (certain Littlewood-Richardson tableaux in her notation). 
%
%
\end{rmk}

\section{RSK correspondences of type A, C}

%
%

In this section, we establish 
RSK correspondence for King tableaux, one of our main results. 
Before the main discussion, we 
briefly review type A RSK correspondence.

\subsection{type A RSK\ok}

\begin{defn}
A \kyo{Knuth array} $w$ of length $n$ is 
a two-line array 
\[
w=
\left(\begin{array}{ccccc}
	u_{1}&\cd&u_{n}   \\
	v_{1}&\cd&v_{n}
\end{array}\right)
\]
such that $u_{j}\in [l], v_{j}\in [k]$ with the following 
lexicographic conditions.
\begin{enumerate}
	\item $u_{j}\le u_{j+1}$.
	\item $u_{j}=u_{j+1}\then v_{j}\le v_{j+1}$.
\end{enumerate}
Denote the set of such arrays by 
$\A_{k}^{l}(n)$. Further, set $\A_{k}^{l}=\cup_{n\ge0}\A_{k}^{l}(n)$. 
Say a Knuth array of length $n$ is a word if 
its top row is $1, \ds, n$.
\end{defn}

\begin{rmk}
Given $w\in \A_{k}^{l}$, 
Let $A(i, j)$ be the number of 
pairs 
$\left(\begin{smallmatrix}	i\\
	j
\end{smallmatrix}\right)$ in $w$. 
Together, $(A(i ,j))$ forms an $l$ by $k$ matrix.
Thus, we can naturally identify 
$\A_{k}^{l}$ and 
$M_{k}^{l}(\zz_{\ge0})$, the set of $l$ by $k$ 
matrices.
\end{rmk}

For $w$ as above, 
start with 
$(P_{1}, Q_{1})=
\k{\fb{$v_{1}$}\,, \fb{$u_{1}$}}$. Let 
$P_{j}=P_{j-1}\from v_{j}$ 
and 
$Q_{j}$ be the SSYT by adding $\fb{$u_{j}$}$ 
to $Q_{j-1}$ at position of the box $P_{j}\sm P_{j-1}$.
Finally, set $(P(w), Q(w))=(P_{n}, Q_{n})$.

\begin{thm}[type A RSK correspondence]
\[
\Phi=\Phi_{A}:\A_{k}^{l}\to \bigcup_{\el(\l)\le k, l} \ssyt_{k}(\l)\ti \ssyt_{l}(\l)
\]
\[
w\mto (P(w), Q(w))
\]
is a bijection.
\end{thm}

One application of this correspondence is \eh{Cauchy identity} as we are going to see. Let $\x=(x_{1}, \ds, x_{k})$
and $\y=(y_{1}, \ds, y_{l})$ 
be variables. 
For a tableau $T$ of entries $[k]$, 
recall that $T(m, n)$ is the entry of $T$ at row $m$ and column $n$. Define \kyo{weight} of $T$ as
\[
\x^{T}=\prod_{m, n} x_{T(m, n)}.
\]
\newcommand{\tes}{\te{s}}
\begin{defn}
The \kyo{Schur polynomial} 
for $\l$ is 
\[
\tes_{\l}(\x)=
\sum_{T\in \ssyt_{k}(\l)}\x^{T}.
\]
For $w\in \A_{k}^{l}$, define its 
\kyo{type A weight} as $m_{A}(w)=\x^{P_{A}(w)}\y^{Q_{A}(w)}$. 
\end{defn}

\begin{fact}
\[
\prod_{i=1}^{l}
\prod_{j=1}^{k}(1-x_{j}y_{i})^{-1}
=
\sum_{\l}\tes_{\l}(\x)\tes_{\l}(\y).
\]
\end{fact}
It follows from the RSK correspondence that 
this infinite product is the generating function of 
the weight.
\[
\prod_{i=1}^{l}
\prod_{j=1}^{k}(1-x_{j}y_{i})^{-1}
=
\sum_{\l}\tes_{\l}(\x)\tes_{\l}(\y)
\]
\[=
\sum_{\l}\sum_{(T, S)\in \ssyt_{k}(\l)\ti \ssyt_{l}(\l)
}\x^{T}\y^{S}
=\sum_{w\in \A_{k}^{l}}m_{A}(w).
\]
Now, let us consider some special case 
$k\mto \ol{k}, l\mto k$. 
Then for each $w\in \A_{\olk}^{k}$, we have 
\[
(P_{A}(w), Q_{A}(w))\in \ssyt_{\olk}(\l)\ti \ssyt_{k}(\l)
\]
for a partition $\l$ (and necessarily $\el(\l)\le k$). Using variables 
$x_{1}, \ds, x_{\olk}$ with specialization 
$x_{\ol{j}}=x_{j}^{-1}$, 
$\x^{\pm}=(x_{1},\ds, x_{k}^{-1})$
and $\y=(y_{1}, \ds, y_{k})$, we have 
\[
\prod_{i, j=1}^{k}(1-x_{j}y_{i})^{-1}
(1-x_{j}^{-1}y_{i})^{-1}
=
\sum_{\el(\l)\le k}\tes_{\l}(\x^{\pm})\tes_{\l}(\y)
=
\sum_{w\in \A_{\olk}^{k}}m_{A}(w).
\]
In fact, there is another way to expand 
this infinite product by type C RSK correspondence. Again, we need an appropriate 
$Q$-symbol for this idea as we will explain with details.

\subsection{semistandard oscillating tableaux\ok}


A skew shape is a \kyo{horizontal strip} if it contains at most one box at each column.
\begin{ex}
Gray boxes indicate it.
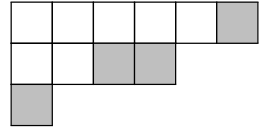
\ytableausetup{mathmode, boxsize=15pt}
\begin{ytableau}
	\mb{}&\mb{}&\mb{}&\mb{}&\mb{}&\lg{}\\
	\mb{}&\mb{}&\lg{}&\lg{}\\
	\lg{}
\end{ytableau}	
\end{ex}

%
%

%

\begin{defn}
The second kind of a $k$-\kyo{semistandard oscillating tableau} ($k$-SSOT) of the final shape $\l$ is a finite sequence $S=(S^{1}, S^{'2}, S^{2}, S^{'3}, \ds, S^{' k}, S^{k})$ of partitions such that 
\begin{enumerate}
\item $S^{i}\supseteq S^{'i+1}$ $(1\le i\le k-1)$
and $S^{'i}\sub S^{i}$ ($2\le i\le k$). Moreover, 
$S^{i} \backslash S^{'i+1}$, 
$S^{i} \backslash S^{'i}$ 
are horizontal strips (possibly empty).
\item $S^{k}=\l$.
\end{enumerate}
Denote by $\ssot_{k}(\l)$ the set of all such SSOTs.
\end{defn}

More intuitively, 
this is a sequence of 
partitions starting at $S^{'1}=\ku$ and ending at $\l$
such that 
at $S^{i}\supseteq S^{'i+1}$, as the first part of a step $i$, we delete several boxes so that those form a horizontal strip and at $S^{'i}\sub S^{i}$, as the second part of the step $i$, we add several boxes similarly. It is a good idea to think that we either add or delete one box at one substep. At an addition step, we add boxes from \kyo{left} and 
at a deletion step, we delete them from \kyo{right}. 
Each partition in such $S$ must have at most $k$ rows since at each step $S^{'i}\sub S^{i}$, 
we add at most one box to columns. 
\begin{rmk}
This idea is based on S.J.Lee \cite{lee}. 
However, this definition is slightly different from his. 
He allows it to end such a sequence with deleting boxes at the final step; taking his to be the first kind, let us call ours the second kind to avoid  unnecessary confusion. The purpose of this modification is to 
construct the bijective RSK correspondence with Berele insertion. We will see in Lemma \ref{lsun} that 
Berele deletions occur only on an \kyo{initial segment} 
of $i$'s in the top row of a lexicographic array.
\end{rmk}

\begin{ex}\label{ex1}
\[
\ytableausetup{mathmode, boxsize=15pt}
S=
\k{
\,\yd{2}\,, 
\yd{1}\,, 
\yd{3,1}\,, 
\yd{1,1}\,,
\yd{2,1}\,
}
\]
is a 3-SSOT.
\end{ex}

There is a convenient way 
to compactly express an SSOT $S$ within one multiset-valued tableau. 
Every time a box is added or deleted at a substep of $S$, we record its \kyo{step number} into a tableau (of possibly larger than the final shape) at the same position. 
For example, we can encode 
the above $S$ as 
\[
\ytableausetup{mathmode, boxsize=15pt}
\k
{\,\begin{ytableau}
	1&1
\end{ytableau}\,, 
\begin{ytableau}
	\mb{}&\lg{2}
\end{ytableau}\,, 
\begin{ytableau}
	\mb{}&2&2\\2
\end{ytableau}\,, 
\begin{ytableau}
	\mb{}&\lg{3}&\lg{3}\\
	\mb{}
\end{ytableau}\,, 
\begin{ytableau}
	\mb{}&3\\
	\mb{}
\end{ytableau}\,}.
\]
All together, we identify $S$ with 
\[
\ytableausetup{mathmode, boxsize=38pt}
\begin{ytableau}
	1&12233&{23}\\
	2\\
\end{ytableau}
\]
as we just write down elements of a multiset in a box; the notation $S(l, m)$ means the multiset 
at $(l, m)$ position in $S$. 
The total number $n$ of letters in an SSOT $S$ is its \eh{length}. It is essential to understand that there is always the canonical numbering on those letters, say $(u_{1}, \ds, u_{n})$, $u_{1}\le \cd \le u_{n}$ and 
$u_{j}\in S(l, m)$ if we either add or delete 
a box at $(l, m)$ position at a substep $j$ in a step $u_{j}$. 
Write $(S)=(u_{1}, \ds, u_{n})$ and call it the \kyo{content} of $S$.
In the above example, 
\[
(S)=
(u_{1}, \ds, u_{9})=
(1, 1, 2, 2, 2, 2, 3, 3, 3).
\]


Notice that an OT of length $n$ can be expressed as an SSOT with its content $(1, \ds, n)$; in this sense, an OT was indeed a \kyo{standard} OT.
For example, 
\[
\ytableausetup{mathmode, boxsize=15pt}
S=
\k{\,
\yd{1}\,, 
\yd{2}\,, 
\yd{1}\,, 
\yd{1,1}\,, 
\yd{2,1}\,,
\yd{3,1}\,,
\yd{2,1}\,,
\yd{1,1}\,,
\yd{2,1}\,
}
\]
\[
\te{corresponds to}\q
\ytableausetup{mathmode, boxsize=38pt}
\begin{ytableau}
	1&23589&{67}\\
	4\\
\end{ytableau}.
\]



%

\subsection{type C RSK}

\begin{defn}
From 
\[
w=
\left(\begin{array}{ccccc}
	u_{1}&\cd&u_{n}   \\
	v_{1}&\cd&v_{n}
\end{array}\right)
\in \A_{\olk}^{l},
\]
we construct a pair of tableaux of the same shape 
\[
(P(w), Q(w))\in \text{KT}_{k}(\l)\ti \te{SSOT}_{ln}(\l)
\]
for some $\l$ ($\el(\l)\le k, l$) as follows. 
Start with 
$(P_{0}(w), Q_{0}(w))=
\k{\ku, \ku}$.
For $j\ge 1 $, let 
\[
P_{j}(w)=P_{j-1}(w)\from v_{j}
\qte{(Berele insertion)}
\]
and 
$Q_{j}(w)$
be the multiset-valued tableau with placing $u_{j}$ into the box which we added or deleted at $P_{j-1}(w)\from v_{j}$. Finally, set 
$(P_{C}(w), Q_{C}(w))=(P_{n}(w), Q_{n}(w))$ 
as we will verify in Lemma \ref{lsun} that 
$Q_{C}(w)$ is an SSOT.
Call $P_{C}(w)$ its \kyo{inserting tableau} and 
$Q_{C}(w)$ \kyo{recording tableau}. 
\end{defn}
\begin{ex}
Let $k=2, l=5, n=11$,
\[\ytableausetup{mathmode, boxsize=18pt}
w=
\left(\begin{array}{ccccccccccc}
	1&1&1&2&\bf{3}&\bf{3}&\bf{4}&4&4&\bf{5}&5   \\
	\oli&2&\olii&2&1&\oli&1&1&\oli&1&\ol{2}
\end{array}\right).
\]
We find that
\[
\begin{ytableau}
	\oli
\end{ytableau}
\to 
\begin{ytableau}
	\oli&2
\end{ytableau}
\to 
\begin{ytableau}
	\oli&2&\olii
\end{ytableau}
\to
\begin{ytableau}
	\oli&2&2\\
	\olii
\end{ytableau}
\to
\begin{ytableau}
	2&2\\
	\olii
\end{ytableau}
\]
\[
\to 
\begin{ytableau}
	\oli&2\\
\end{ytableau}
\to 
\begin{ytableau}
	2
\end{ytableau}
\to
\begin{ytableau}
	1\\
	2
\end{ytableau}
\to 
\begin{ytableau}
	1&\oli\\
	2
\end{ytableau}
\to 
\begin{ytableau}
	1\\
	2
\end{ytableau}
\to 
\begin{ytableau}
	1&\ol{2}\\
	2
\end{ytableau}
\]
as bold letters at the top row indicate steps for deletion.
Thus, 
\[
\ytableausetup{mathmode, boxsize=38pt}
(P(w), Q(w))
=
\k{\,
\begin{ytableau}
	1&\ol{2}\\
	2
\end{ytableau}
, \,\,
\begin{ytableau}
	1&14455&13\\
	234
\end{ytableau}
\,}.
\]
\end{ex}

\begin{lem}\label{lho}
Let $T$ be a KT and $x=x_{1}\cd x_{n}$ be a row word 
$(i.e., x_{j}\le x_{j+1})$. Consider 
successive Berele insertions 
\[
T\from x_{1} \from \cd \from x_{n}.
\]
\begin{enumerate}
\item Suppose each 
insertion $\from x_{j}$ adds a box. 
Then all added boxes form a horizontal strip (appearing from \kyo{left}).
\item Suppose each 
insertion $\from x_{j}$ deletes a box. 
Then all deleted boxes form a horizontal strip 
(disappearing from \kyo{right}).
\item Suppose 
$\from x_{n}$ deletes a box. 
Then, in fact, so do all $\from x_{j}$ $(1\le j\le n-1)$.
\end{enumerate}
\end{lem}
\begin{proof}
This is a consequence of 
\cite[Lemma 3.15, 10.4, 10.3]{sun1}.
\end{proof}

\begin{lem}\label{lsun}
Let 
$w=
\left(\begin{array}{ccccc}
	u_{1}&\cd&u_{n}   \\
	v_{1}&\cd&v_{n}
\end{array}\right)\in \A_{\olk}^{l}$.
Suppose 
$u_{p-1}<u_{p}=\cd =u_{r}=i<u_{r+1}$
and 
consider the successive insertions 
\[
((P(v_{1}\cd v_{p-1})\from v_{p})\from \cd )\from v_{r}.
\]
Then, exactly one of the following happens:
\begin{enumerate}
\item For each $j$ ($p\le j\le r$), 
$P(v_{1}\cd v_{j-1})\from v_{j}$ adds a box.
\item There exists some $q$ $(p\le q\le r)$ such that 
each 
\[
P(v_{1}\cd v_{j-1})\from v_{j}
\]
deletes a box for $p\le j\le q$ and 
\[
P(v_{1}\cd v_{j-1})\from v_{j}
\]
adds it for $q+1\le j\le r$. 
\end{enumerate}
Moreover, all deleted boxes in the step $i$ 
form a horizontal strip in $Q(v_{1}\cd v_{q})$ 
while added boxes also form it in $Q(v_{1}\cd v_{r})$. 
As a consequence, $Q_{C}(w)$ is an 
$l$-SSOT of length $n$. 
\end{lem}
\begin{proof}
This follows from Lemma \ref{lho}.
\end{proof}

\begin{defn}
By $\ssot_{k,n}(\l)$ we mean the set of $k$-SSOTs of shape $\l$, length $n$. 
Let 
$S\in \ssot_{k,n}(\l)$ with $(S)=(u_{1}, \ds, u_{n})$. 
 Define its \kyo{standardization} $S^{\st}\in \ot_{k,n}(\l)
$ by replacing $u_{j}$ by $j$. 
\end{defn}
\begin{ex}
\ytableausetup{mathmode, boxsize=28pt}
\[
S=
\begin{ytableau}
	1&1&12\\
	2&344\\
	3\\
\end{ytableau}, \q
S^{\st}=
\begin{ytableau}
	1&2&34\\
	5&789\\
	6\\
\end{ytableau}. 
\]
\end{ex}

\begin{defn}
Let $w=
\left(\begin{array}{ccccc}
	u_{1}&\cd&u_{n}   \\
	v_{1}&\cd&v_{n}
\end{array}\right)\in \A_{\olk}^{l}(n)$. 
Its \kyo{standardization} $w^{\st}$ 
is 
\[
w^{\st}=
\left(\begin{array}{ccccc}
	1&\cd&n   \\
	v_{1}&\cd&v_{n}
\end{array}\right) \in \A^{n}_{\olk}.
\]
\end{defn}
We often identify 
such an array $w^{\st}$ with a word $v_{1}\cd v_{n}$.

\begin{thm}[RSK correspondence of type C]
\label{t5}
\[
\Phi=\Phi_{C}:\A_{\olk}^{k}\to 
\bigcup_{\el(\l)\le k}
\te{KT}_{k}(\l)\ti \te{SSOT}_{k}(\l),
\q w\mto (P_{C}(w), Q_{C}(w))
\]
is a bijection.
\end{thm}

\begin{proof}
It is enough to show that 
for each integer $n\ge1$,
\[
\Phi:\A_{\olk}^{k}(n)\to 
\bigcup_{{\el(\l)\le k
}
}
\te{KT}_{k}(\l)\ti \te{SSOT}_{kn}(\l),
\q 
w\mto (P_{C}(w), Q_{C}(w))
\]
is a bijection.

If $n=1$, then 
$w=
\left(\begin{array}{ccccc}
	u_{1}\\
v_{1}
\end{array}\right)$, 
$\Phi(w)=\k{\fb{$v_{1}$}, \fb{$u_{1}$}}$ is certainly a bijection.
Suppose $n\ge2$. 
Let 
\[
w=
\left(\begin{array}{ccccc}
	u_{1}&\cd&u_{n}   \\
	v_{1}&\cd&v_{n}
\end{array}\right), 
w'=
\left(\begin{array}{ccccc}
	u_{1}'&\cd&u_{n}'   \\
	v_{1}'&\cd&v_{n}'
\end{array}\right)\in \A_{\olk}^{k}(n).
\]
For convenience, let 
\[
w_{0}=
\left(\begin{array}{ccccc}
	u_{1}&\cd&u_{n-1}   \\
	v_{1}&\cd&v_{n-1}
\end{array}\right), 
w_{0}'=
\left(\begin{array}{ccccc}
	u_{1}'&\cd&u_{n-1}'   \\
	v_{1}'&\cd&v_{n-1}'
\end{array}\right),
\]
$v=v_{1}\cd v_{n} $
and $v'=v'_{1}\cd v'_{n} $. To show that $\Phi$ is injective, assume 
$\Phi(w)=\Phi(w')$.
Then, in particular, $Q(w)=Q(w')$.
Now look at its box which contains $u_{n}$ and $u_{n}'$. These are maximal entries of this box 
so that $u_{n}=u_{n}'$.
Removing it from $Q(w)=Q(w')$, we see $Q(w_{0})=Q(w_{0}')$.
With the same algorithm, we can show that 
$u_{n-1}=u_{n-1}'$ and so on. Thus, 
\[
(u_{1}, \ds, u_{n})=
(u'_{1}, \ds, u'_{n}).
\]
Now $Q(w)=Q(w')$ implies 
$Q(w)^{\st}=Q(w')^{\st}$ so that 
\[
\Phi_{\te{Ber}}(v)
=
(P(w^{\st}), Q(w^{\st}))
=
(P(w), Q(w)^{\st})
\]
\[=
(P(w'), Q(w')^{\st})
=
(P(w'^{\st}), Q(w'^{\st}))
=\Phi_{\te{Ber}}(v').
\]
Berele's RS correspondence shows $v=v'$ and hence 
$w=w'$.

It remains to show that 
$\Phi$ is surjective.
Let $(T, S)\in\kt_{k}(\l)\ti \ssot_{kn}(\l)$ 
and $(S)=(u_{1}, \ds, u_{n})$.
Consider $S^{\st}\in \ot_{kn}(\l)$. For a pair $(T, S^{\st})$, there exists a unique $v=v_{1}\ds v_{n}\in [\olk]^{*}$ by Berele's RS correspondence 
such that 
\[
(P(v), Q(v))=(T, S^{\st})\in \kt_{k}(\l)\ti \te{OT}_{kn}(\l).
\]
Now define 
\[
w=
\left(\begin{array}{ccccc}
	u_{1}&\cd&u_{n}   \\
	v_{1}&\cd&v_{n}
\end{array}\right)
\]
so that $Q(w)=S$.
Conclude that 
\[
\Phi(w)=(P(w), Q(w))=
(P(v), Q(w))=(T, S).
\]
\end{proof}

\begin{rmk}
In fact, this is essentially a reformulation of Sundaram's discussion  \cite[p.120-127]{sun1}. Her correspondence deals with certain triples including a Burge array or a  lattice permutation, though.
\end{rmk}

\subsection{duality of Cauchy identity\ok}

We have seen before the expansion of 
Cauchy identity
\[
\prod_{i, j=1}^{k}(1-x_{j}y_{i})^{-1}
(1-x_{j}^{-1}y_{i})^{-1}
=
\sum_{\el(\l)\le k}\tes_{\l}(\x^{\pm})\tes_{\l}(\y).
\]
Let us now see hidden duality of this product 
as a consequence of 
our RSK correspondence. 
\begin{defn}
For $T\in \kt_{k}(\l)$, define 
$\x^{T}=\prod_{m, n}x_{T(m, n)}$ 
where $x_{\ol{j}}=x_{j}^{-1}$ and 
\kyo{symplectic Schur function (Laurent polynomial)} by 
\[
\tsp_{\l}(\x^{\pm})=
\sum_{T\in \kt_{k}(\l)}\x^{T}.
\]
\end{defn}

\begin{defn}
Let $\l$ be a partition, $n\ge |\l|$ 
such that $\el(\l)\le k$, $\hou{n}{|\l|}{2}.$ 
For $S\in \ssot_{k,n}(\l)$, 
recall that $S(l, m)$ is the multiset in $S$ at row $l$ and column $m$ and it consists of unbarred letters.
Define its \kyo{weight}
\[
\y^{S}=
\prod_{l, m}\prod_{j\in S(l, m)}y_{j}.
\]
This is a monomial of degree $n$. Define the \kyo{SSOT polynomial} of $\l$ of degree $n$ as 
\[
\ss_{\l, n}(\y)=
\sum_{S\in \ssot_{k,n}(\l)}\y^{S}
\]
and the \kyo{SSOT function (formal power series)} is 
\[
\ss_{\l}(\y)=
\sum_{S\in \ssot_{k}(\l)}\y^{S}.
\]
\end{defn}
\begin{rmk}
If $n=|\l|$, 
$\ssot_{k, |\l|}(\l)=\ssyt_{k}(\l)$ so that 
\[
\ss_{\l,|\l|}(\y)=\tes_{\l}(\y).
\]
\end{rmk}

For an SSOT $S$, 
let $S^{\red}$, its \kyo{reduction}, be the SSYT 
with only maxima of boxes on its final shape. 
Then $S_{1}\sim S_{2}\iff 
S_{1}^{\red}=S_{2}^{\red}$ naturally defines an equivalent relation 
on $\ssot_{k,n}(\l)$.
\begin{ex}
Let $\l=(2, 1)$, 
$k=3$. 
In fact, all eight SSYTs of $\l$ give 
$\tes_{\l}(\y)=(y_{1}+y_{2})(y_{1}+y_{3})(y_{2}+y_{3})$
and 
\[
\tes_{\l, 5}(\y)=
(y_{1}y_{2}+y_{1}y_{3}+y_{2}y_{3})
\tes_{\l}(\y)
\]
as we see all 24 SSOTs of 
length 5 of shape $\l$ (with eight equivalent classes) below.
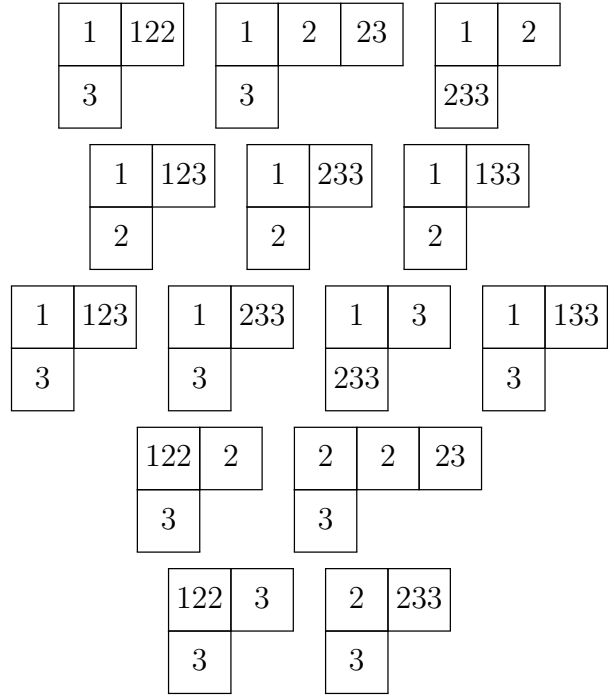
\[
\ytableausetup{mathmode, boxsize=23pt}
\mb{}\hf
\begin{ytableau}
	1&1&12\\2\\
\end{ytableau}
\q
\mb{}
\begin{ytableau}
	1&1&13\\2\\
\end{ytableau}
\q
\begin{ytableau}
	1&1&23\\2\\
\end{ytableau}
\q
\begin{ytableau}
	1&1\\2&23\\
\end{ytableau}
\q
\]
\[
\begin{ytableau}
	1&1&12\\3\\
\end{ytableau}
\q
\begin{ytableau}
	1&1&13\\3\\
\end{ytableau}
\q
\begin{ytableau}
	1&1&23\\3\\
\end{ytableau}
\q
\begin{ytableau}
	1&1\\233\\
\end{ytableau}
\]
\[
\begin{ytableau}
	1&122\\2\\
\end{ytableau}
\q
\begin{ytableau}
	1&2&23\\2\\
\end{ytableau}
\]
\[
\begin{ytableau}
	1&122\\3\\
\end{ytableau}
\q
\begin{ytableau}
	1&2&23\\3\\
\end{ytableau}
\q
\begin{ytableau}
	1&2\\233\\
\end{ytableau}
\]
\[
\begin{ytableau}
	1&123\\2\\
\end{ytableau}
\q
\begin{ytableau}
	1&233\\2\\
\end{ytableau}
\q
\begin{ytableau}
	1&133\\2\\
\end{ytableau}
\]
\[
\begin{ytableau}
	1&123\\3\\
\end{ytableau}
\q
\begin{ytableau}
	1&233\\3\\
\end{ytableau}
\q
\begin{ytableau}
	1&3\\233\\
\end{ytableau}
\q
\begin{ytableau}
	1&133\\3\\
\end{ytableau}
\]
\[
\begin{ytableau}
	122&2\\3\\
\end{ytableau}
\q
\begin{ytableau}
	2&2&23\\3\\
\end{ytableau}
\]
\[
\begin{ytableau}
	122&3\\3\\
\end{ytableau}
\q\begin{ytableau}
	2&233\\3\\
\end{ytableau}
\]
\end{ex}

\begin{cor}[duality of Cauchy identity]\label{c1}
\[
\sum_{\el(\l)\le k}
\tes_{\l}(\x^{\pm})\tes_{\l}(\y)=
\prod_{i, j=1}^{k}(1-x_{j}y_{i})^{-1}
(1-x_{j}^{-1}y_{i})^{-1}
=\sum_{\el(\l)\le k}
\tsp_{\l}(\x^{\pm})\ss_{\l}(\y).
\]
\end{cor}
\begin{proof}
\[
\te{LHS}=
\sum_{\el(\l)\le k}
\sum_{(T, S)\in \ssyt_{\olk}(\l)\ti 
\ssyt_{k}(\l)}\x^{T}\y^{S}
\]
\[=
\sum_{A\in M_{\olk}^{k}(\zz_{\ge0})}
\prod_{i, j}(x_{j}y_{i})^{A(i, j)}
=
\sum_{\el(\l)\le k}
\sum_{(T, S)\in \kt_{k}(\l)\ti 
\ssot_{k}(\l)}\x^{T}\y^{S}
=\te{RHS}.
\]
\end{proof}

\begin{rmk}
Jae-Hoon Kwon kindly pointed out (private communication) that 
we can also show this Cauchy type identity 
as a special case of \cite[Theorem 7.10]{kwon} 
with a symplectic analog of RSK for a spinor model.
\end{rmk}

%
%
%
%

\subsection{Bender-Knuth involution for SSOTs}

As Bender-Knuth proved \cite{bek}, all $s_{\l}(\x)$ are symmetric. 
Their idea was to construct certain involution  (\kyo{Bender-Knuth involution}) as we explain  here. Fix a letter $i$. Say a box $\fb{$i$}$ $\k{\fb{$i+1$}}$ 
in an SSYT $T$ 
is \kyo{frozen} 
if 
there is \fb{$i+1$} $\k{\fb{$i$}}$ 
right below (above) it.
Non-frozen boxes 
$\fb{$i$}, \fb{$i+1$}$
are \kyo{mutable}.
In each row $r$ of $T$, 
letters in mutable boxes form a sequence 
$i^{m_{r}}(i+1)^{n_{r}}$, $m_{r}, n_{r}\ge 0$.
Replacing this by 
$i^{n_{r}}(i+1)^{m_{r}}$ and 
fixing all other entries, we get another tableau (necessarily an SSYT). Apparently, this map is an involutive. 
For example, for $i=3$:
\[
\ytableausetup{mathmode, boxsize=15pt}
\begin{ytableau}
	1&1&1&2&2&2&3&\lg 4\\
	2&2&3&\lg 3&\lg 4&\lg 4&4\\
	\lg 3&\lg 3&4\\
\end{ytableau}
\lr 
\begin{ytableau}
	1&1&1&2&2&2&3&\lg 3\\
	2&2&3&\lg 3&\lg 3&\lg 4&4\\
	\lg 4&\lg 4&4\\
\end{ytableau}
\]
We can now abstractly state their result as follows.
\begin{lem}
Let $1\le i\le k-1$ and $\wt_{i}T$ be the total number of $i$ in a tableau $T$. 
There exists an involution 
\[
f=f_{i}:\ssyt_{k}(\m)\to \ssyt_{k}(\m)
\]
such that 
\[
(\wt_{i}f(T), \wt_{i+1}f(T))
=(\wt_{i+1}T, \wt_{i}T)
\]
and for all $j\ne i, i+1$, $\wt_{j}f(T)=\wt_{j}T$.
\end{lem}

Guti\'{e}rrez (2024) \cite{gu} proved that 
$\tsp_{\l}(\x^{\pm})$ is $B_{2k}$-symmetric where 
$B_{2k}$ means the set of signed permutations:
\[
B_{2k}=\{w:[\olk]\to [\olk]\mid 
w \te{ is a bijection}, w(\ol{i})=\ol{w(i)}\}.
\]
Guti\'{e}rrez's proof is also due to construction of  Bender-Knuth involution for KTs. 

\begin{lem}\label{lab}
Let $1\le i\le k-1$. 
There exists an involution 
\[
g=g_{i}:\ssot_{k,n}(\l)\to \ssot_{k,n}(\l)
\]
such that 
\[
(\wt_{i}g(S), \wt_{i+1}g(S))
=(\wt_{i+1}S, \wt_{i}S)
\]
and $\wt_{j}g(S)=\wt_{j}S$ for all $j\ne i, i+1$.
\end{lem}
To understand our proof, 
it is helpful to keep this diagram in mind:
\begin{center}
\begin{minipage}[c]{.8\tw}
\xymatrix@!=15mm{
*{
\begin{array}{ccc}
\kt_{k}(\l)\ti \ssot_{k,n}(\l)
\end{array}}
\ar@{<->}_{\te{id}\ti g}[rrr]
\ar@{<->}_{\tn{AC RSK}}[d]
&&&
*{
\begin{array}{ccc}
\kt_{k}(\l)\ti \ssot_{k,n}(\l)
\end{array}}
\ar@{<->}_{\tn{AC RSK}}[d]\\
*{
\begin{array}{ccc}
\ssyt_{\olk}(\m)\ti \ssyt_{k}(\m)
\end{array}}
\ar@{<->}_{\te{id}\ti f}[rrr]
&&&
*{
\begin{array}{ccc}
\ssyt_{\olk}(\m)\ti \ssyt_{k}(\m)
\end{array}}
}
\end{minipage}
\end{center}
\begin{proof}
Let 
$S\in \ssot_{k,n}(\l)$ and 
choose any $T\in \kt_{k}(\l)$.
For a pair $(T, S)$, one has 
\[
\Phi_{C}(w)=(T, S)
\]
for a unique $w\in \A_{\olk}^{k}(n)$.
Let 
\[
(P, Q)=\Phi_{A}(w)=
(P_{A}(w), Q_{A}(w)),
\]
say $\sh P=\sh Q=\m$ ($|\m|=n$). 
Let $f=f_{i}$ be the type A Bender-Knuth involution on 
$\ssyt_{k}(\m)$. Here we have 
\[
(\wt_{i}f(Q), \wt_{i+1}f(Q))=
(\wt_{i+1}Q, \wt_{i}Q).
\]
For a pair $(P, f(Q))$, 
find 
a unique $w'\in \A_{\olk}^{k}(n)$ such that 
$\Phi_{A}(w')=(P, f(Q))$.
Then 
$\Phi_{C}(w')=(T, Q_{C}(w'))$, $\sh Q_{C}(w')=\sh T=\l$. 
Define $g(S)=Q_{C}(w')$.
Then 
\[
(\wt_{i}g(S), 
\wt_{i+1}g(S))
=
(\wt_{i}f(Q), 
\wt_{i+1}f(Q))
=(\wt_{i+1}Q, 
\wt_{i}Q)
=
(\wt_{i+1}S, 
\wt_{i}S)
\]
and similarly for all $j\ne i, i+1$, 
\[
\wt_{j}g(S)=\wt_{j}f(Q)=\wt_{j}Q=\wt_{j}S.
\]
Notice that all maps in the diagram are bijections.
Hence, staring from $(T, g(S))$, we can go backward
\[
(T, g(S))
\too
(P, f(Q))
\too 
(P, f(f(Q)))=(P, Q)\too (T, S).\]
This implies $g$ is a bijection and $g(g(S))=S$.
%

%
%
\end{proof}

\begin{lem}\label{lsym}
All $\ss_{\l, n}(\y)$ are symmetric. 
\end{lem}
\begin{proof}
Since transpositions 
$\s_{i}=(i, i+1)$ generate $S_{k}$, 
it is enough to show that for each $i$,
\[
\s_{i}\k{\ss_{\l, n}(\y)}=\ss_{\l, n}(\y)
\]
where $\s_{i}$ flips $y_{i}$ and $y_{i+1}$.
Lemma \ref{lab} implies that 
\[
\s_{i}\k{\ss_{\l, n}(\y)}
=
\s_{i}\k{
\sum_{g(S)\ne S}\y^{S}
+
\sum_{g(S)= S}\y^{S}
}
\]
\[=
\s_{i}\k{
\sum_{g(S)\ne S}\y^{S}
}
+
\s_{i}
\k{\sum_{
\substack{g(S)= S\\
\wt_{i} S=\wt_{i+1}S
}
}\y^{S}
}
=
\sum_{g(S)\ne S}\y^{S}
+
\sum_{g(S)= S}\y^{S}
=\ss_{\l, n}(\y).\]
\end{proof}
\begin{thm}\label{t6}
All $\ss_{\l}(\y)$ are symmetric.
\end{thm}
\begin{proof}
Lemma \ref{lsym} just showed all $\ss_{\l n}(\y)$ are  symmetric for each $n$, $n\ge|\l|, 
\hou{n}{|\l|}{2}$. As a consequence, so is 
\[
\ss_{\l}(\y)=
\sum_{n}\ss_{\l, n}(\y).
\]
\end{proof}

\subsection*{Note added in the proof.} 
After writing this manuscript, 
Hideya Watanabe told us that 
there is a representation-theoretical interpretation 
of our results 
via the quantum symmetric pairs of type AII (not  C); see \cite{wa} for details. 
Thus, our usage of type A \& C was a little abuse of language.

%




\end{document}